\documentclass[12pt]{amsart}
\usepackage{latexsym}
\usepackage{graphicx}
\usepackage{amsmath}
\usepackage{amssymb}
\usepackage{verbatim}
\usepackage{amsfonts}

\newtheorem{lemma}{Lemma}
\newtheorem*{thm}{Theorem}

\newcommand{\h}{\mathbb{H}}

\begin{document}

\title{Thick triangulations of hyperbolic $n$-manifolds}
\author{William G. Breslin}
\date{August 20, 2007}

\begin{abstract}
We show that a complete hyperbolic $n$-manifold has a geodesic triangulation such that the tetrahedra contained in the thick part are $L$-bilipschitz diffeomorphic to the standard Euclidean $n$-simplex, for some constant $L$ depending only on the dimension and the constant used to define the thick-thin decomposition of $M$.
\end{abstract}

\maketitle


A geodesic triangulation of a complete hyperbolic n-manifold $M$ may be forced by the geometry of $M$ to have simplices which are either small or flat.  Big simplices without small dihedral angles cannot live in the thin part of $M$.  We show that a complete hyperbolic $n$-manifold $M$ has a geodesic triangulation such that the simplices contained in the thick part of $M$ are $L$-bilipschitz diffeomorphic to the standard Euclidean $n$-simplex, for some constant $L$ depending only on the dimension $n$ and the constant $\mu$ used to define the thick-thin decomposition of $M$. We call such a triangulation a \textit{$(\mu ,L)$-thick} geodesic triangulation of $M$.

\begin{thm}
Let $n \ge 2$.  Let $\mu$ be a Margulis constant for $\mathbb{H}^n$.  There exists a constant $L:=L(n,\mu)$ such that every complete hyperbolic $n$-manifold has a $(\mu, L)$-thick geodesic triangulation.
\end{thm}

Existence of thick geodesic triangulations implies that any hyperbolic n-manifold $M$ has a geodesic triangulation such that the simplices contained in the thick part of $M$ come from a fixed compact set of simplices which does not depend on the manifold.  In \cite{Breslin}, we use this compactness to prove existence of bounds on the curvatures of surfaces embedded in hyperbolic 3-manifolds.  In particular, it is shown that there exists a fixed constant $\omega > 0$ such that if $S$ is an incompressible surface or a strongly irreducible Heegaard surface in a complete orientable hyperbolic 3-manifold, then $S$ is isotopic to a surface whose principal curvatures are bounded in absolute value by $\omega$.  The constant $\omega$ depends on neither the hyperbolic $3$-manifold nor the surface.  Another interesting application of thick triangulations can be found in \cite{kapovich2}, where Kapovich explains how to use thick triangulations of hyperbolic $n$-manifolds to obtain an inequality between the relative homological dimension of a Kleinian group $\Gamma \subset$ Isom($\h^n$) and its critical exponent.

To prove existence of thick geodesic triangulations, we examine Delaunay triangulations of well-spaced point sets in hyperbolic n-space and the problem of eliminating flat simplices (i.e. simplices with small dihedral angles).  The corresponding question for Euclidean space has been well-studied.  The only tetrahedra in such a triangulation which can have small dihedral angles are called \textit{slivers}, and it was a problem to show how to remove them without creating new ones.  Several techniques for removing slivers have been developed in the Euclidean setting (see \cite{Edels},\cite{Miller},\cite{Li}, \cite{Li2}).  We adapt the technique introduced in \cite{Edels} of perturbing vertices of a Delaunay triangulation in order to remove slivers to the hyperbolic setting.

Emil Saucan has shown that hyperbolic $n$-orbifolds have triangulations whose simplices are uniformly round (called ``fat" triangulations), and he uses this to prove existence of quasi-meromorphic maps which are automorphic with respect to the corresponding Kleinian group (See \cite{saucan1}, \cite{saucan2}, \cite{saucan3}).\\

\noindent Let $M$ be a complete hyperbolic $n$-manifold.\\

\noindent\textbf{Definition} (\textit{thick-thin decomposition}).   Let $\mu > 0$. The $\mu$-\textit{thick part} of $M$, denoted by $M_{[\mu ,\infty )}$ is the set of points where the injectivity radius is at least $\mu / 2$.  The $\mu$-\textit{thin part} of $M$, denoted by $M_{(0, \mu ]}$, is the closure of the complement of $M_{[\mu, \infty)}$.\\

\noindent\textbf{Definition} (\textit{thick triangulation}).  Let $\mu > 0$, $L > 0$. A triangulation $T$ of a complete hyperbolic $n$-manifold $M$ is $(\mu, L)$-\textit{thick} if every $n$-simplex of $T$ which is contained in the $\mu$-thick part of $M$ is $L$-bilipschitz diffeomorphic to the standard Euclidean $n$-simplex.  Once we have fixed $\mu$ and $L$, we will refer a \textit{thick} triangulation.\\

\noindent\textbf{Definition} (\textit{Delaunay triangulation}).  Let $\mathcal{S}$ be a generic set of points in $M$ such that for any $p$ $\in$ $M$ the ball $B(p,inj(M,p)/5)$ centered at $p$ with radius $inj(M,p)/5$ contains a point of $\mathcal{S}$ in its interior.  The \textit{Delaunay triangulation of $\mathcal{S}$} is the geodesic triangulation of $M$ determined as follows:  A set, $\{ p_0, ... , p_{n} \}$, of $n+1$ vertices in $\mathcal{S}$ determines an $n$-simplex in $\mathcal{T}$ if and only if the minimal radius circumscribing sphere contains no points of $\mathcal{S}$ in its interior.\\

See \cite{Leibon} for existence of Delaunay triangulations in Riemannian manifolds.

We want to find triangulations such that the simplices in the  thick part of $M$ are neither too big nor too small, and which do not have small dihedral angles.  It is not difficult to find triangulations such that the simplices in the thick part are neither too big nor too small.  Let $\mu > 0$ be a Margulis constant for $\mathbb{H}^n$.  Let $\epsilon := \mu /100$.  Let $\mathcal{S}$ be a generic set of points in $M$ such that for any $p$ $\in$ $M$ the ball $B(p,inj(M,p)/5)$ centered at $p$ with radius $inj(M,p)/5$ contains a point of $\mathcal{S}$ in its interior.  Also assume that the set $\mathcal{S}$ is maximal with respect to the condition that each point in $\mathcal{S} \cap M_{[\mu ,\infty)}$ is no closer than $\epsilon$ to another point of $\mathcal{S}$.  Let $T$ be the Delaunay triangulation of $\mathcal{S}$.  Any simplex of $T$ in the $\mu$-thick part of $M$ has edge lengths in the interval $[\epsilon ,2\epsilon]$.  In fact, this triangulation is not very far from the one we want.  We will show that each vertex of $\mathcal{S} \cap M_{[\mu,\infty)}$ can be moved a small distance so that the Delaunay triangulation of the new set of points is $(\mu, L)$-thick.\\

\noindent\textbf{Definition} (\textit{altitude}).  The \textit{altitude} of a vertex $v$ of a geodesic $n$-simplex in $\h^n$ is the distance from $v$ to the hyperplane of $\h^n$ containing the other vertices.\\

\noindent\textbf{Definition} (\textit{good simplices}). For $2 \le k \le n$, $0 < a < b$, and $0 < d$, a geodesic $k$-dimensional simplex $S$ in hyperbolic $n$-space $\mathbb{H}^n$ is \textit{(a,b,d)-good} if the lengths of the edges of $S$ are contained in the interval $[a,b]$ and the altitude of each vertex of $S$ is at least $d$.  When $a,b,d$ are understood, we will refer to \textit{good simplices}.  We say $S$ is \textit{bad} if it is not good.\\


\noindent\textit{Remark 1.}  If a geodesic $n$-simplex $S$ in $\h^n$ has edge lengths in $[a,b]$, then there are two ways that it can be $(a,b,d)$-bad for a small number $d > 0$.  Either $S$ has big circumradius or the vertices of $S$ all lie near a hyperbolic $(n-2)$-sphere.  In the triangulation $T$ described above the simplices in $T \cap M_{[\mu ,\infty)}$ have bounded circumradii, so that the vertices of any $(a,b,d)$-bad simplex in $T \cap M_{\mu ,\infty)}$ must all lie close to a hyperbolic $(n-2)$-sphere. The vertices get closer to an $(n-2)$-sphere as $d\rightarrow 0$.\\


\noindent\textit{Remark 2.}  Let $b > a > 0$ and $d > 0$.  Consider the set of compact hyperbolic $n$-simplices in $\h^n$ up to isometry.  The set of geodesic $(a,b,d)$-good simplices is a compact subset.  Thus we have the following lemma.

\begin{lemma}\label{lem1}
For each $n \ge 2$, $b > a > 0$, and $d > 0$, there exists a constant $L:=L(n,a,b,d)$ such that each $(a,b,d)$-good simplex is $L$-bilipschitz diffeomorphic to the standard Euclidean $n$-simplex.
\end{lemma}

\noindent\textbf{Definition} (\textit{good perturbation}).  Let $\delta > 0$.  A $\delta$\textit{-good perturbation} of $\mathcal{S}$ is a collection of points $\mathcal{S'}$ in $M$ such that there exists a bijection $\phi : \mathcal{S} \rightarrow \mathcal{S'}$ with $d(p,\phi(p)) \le \delta$ for every $p \in \mathcal{S}$. Denote $\phi(p)$ by $p'$.  If $\mathcal{T}$ and $\mathcal{T'}$ are the Delaunay triangulations of $\mathcal{S}$ and $\mathcal{S'}$, then we will say that $\mathcal{T'}$ is a $\delta$\textit{-good perturbation} of $\mathcal{T}$.  When $\delta$ is understood, we will refer to a \textit{good} perturbation.\\

\noindent\textbf{Definition} (\textit{bad region}).  Let $S =[v_1, ... ,v_k]$ be a geodesic $(k-1)$-simplex in $\mathbb{H} ^n$. Let $b > a > 0$, $c > 0$, $d >0$.   The $(a,b,c,d)$-\textit{bad region} of $S$ is the set of points $p$ in $\mathbb{H} ^3$ such that $[p,v_1 , ... ,v_k]$ has edge lengths in $[a,b]$, circumradius at most $c$, and the distance from $p$ to the hyperplane containing the opposite face is less than $d$. \\

The next lemma shows if a bad simplex has good proper sub-simplices, then each vertex is close to the plane containing the other vertices.

\begin{lemma}\label{lem2}
Let $S = [v_0 , ... ,v_k]$ be a geodesic $k$-simplex in $\mathbb{H}^k$ with edge lengths in $[a,b]$.  If every proper subcomplex of $S$ is $(a,b,d_0 )$-good and $S$ is $(a,b,d)$-bad, then the distance from each vertex of $S$ to the hyperplane containing the opposite face is at most a constant $D := D(b,d_0 ,d)$ such that $D(b,d_0 ,d) \rightarrow 0$ as $d \rightarrow 0$ and $b$ and $d_0$ remain fixed.
\end{lemma}

\begin{figure}[h]
\includegraphics[width=\textwidth]{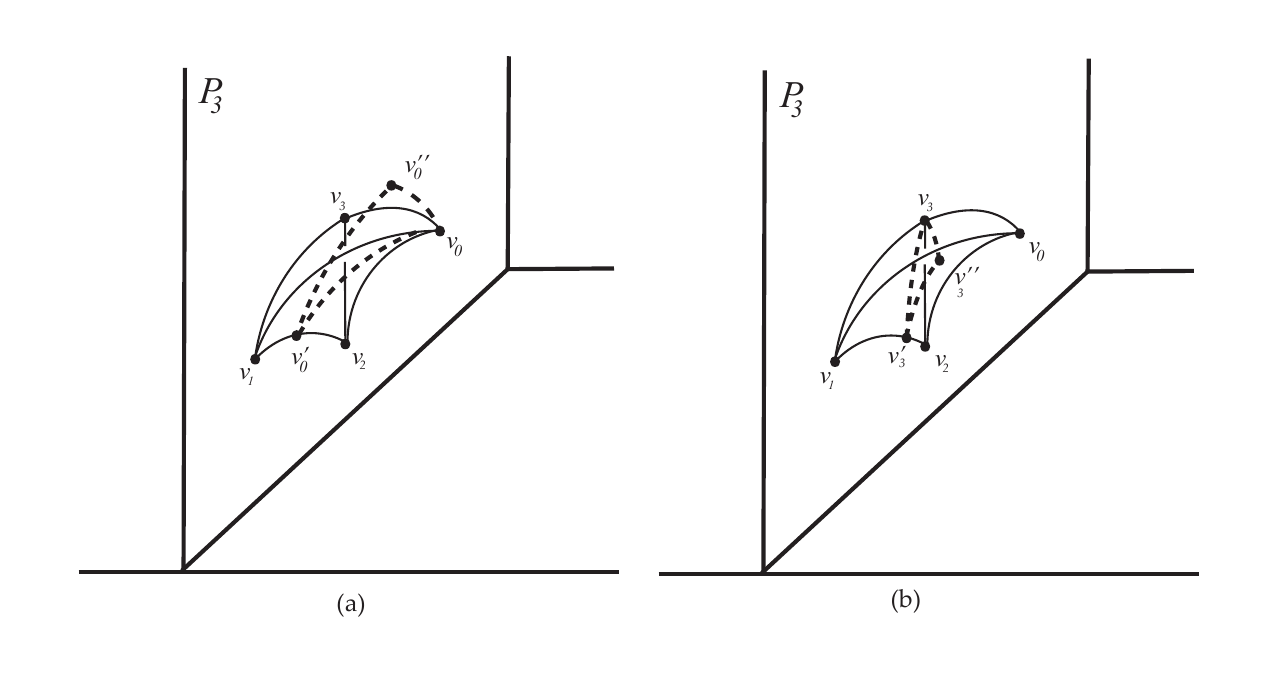}
\caption{If every proper subcomplex of $S$ is good and $S$ is bad, then each vertex is close to the plane containing the opposite face.}
\label{fig1}
\end{figure}

\begin{proof}
Since $S$ is $(a,b,d)$-bad and the edge lengths are in $[a,b]$, the distance from some vertex, say $v_0$, to the hyperplane in $\mathbb{H}^k$ containing the opposite face $[v_1 , ... ,v_k ]$ is less than $d$.  Let $P_0$ be the hyperplane containing $v_0 , ... ,v_{k-1}$.  Let $P_k$ be the hyperplane containing $v_1 , ... ,v_k$.   Let $\alpha$ be the angle between $P_0$ and $P_k$.  Let $v_0 '$ be the orthogonal projection of $v_0$ to $P_0 \cap P_k$ and let $v_0 ''$ be the orthogonal projection of $v_0$ to $P_k$.  The angle of the hyperbolic triangle $[v_0 ,v_0 ' ,v_0 '']$ at $v_0 '$ is $\alpha$.  See Figure \ref{fig1}(a).  The hyperbolic law of sines \cite{Fenchel} gives us

\begin{center}
$\sin(\alpha) = \frac{\sinh(||[v_0 ,v_0 '']||)}{\sinh(||[v_0 ,v_0 ']||)}$.
\end{center}

Let $v_k '$ be the orthogonal projection of $v_k$ to $P_0 \cap P_k$ and let $v_k ''$ be the orthogonal projection of $v_k$ to $P_0$.  The angle the hyperbolic triangle $[v_k, v_k ' ,v_k '']$ at $v_k ''$ is also $\alpha$.  See Figure \ref{fig1}(b). Using the hyperbolic law of sines again we get

\begin{align*}
\sinh(||[v_k ,v_k '']||) &= \sin(\alpha)\cdot\sinh(||[v_k ,v_k ']||)\\
                          &= \frac{\sinh(||[v_0 ,v_0 ''])}{\sinh(||[v_0 ,v_0 ']||)} \cdot \sinh(||[v_k ,v_k ']||).
\end{align*}

Now $||[v_0 ,v_0 ']|| \ge d_0$ and $||[v_k ,v_k ']|| \le b$ since $||[v_0 , ... ,v_{k-1}]||$ and  $||[v_1 , ... ,v_k]||$ are $(a,b,d_0)$-good $(k-1)$-simplices.   Also, $||[v_0 ,v_0 '']|| < d$ by our assumption.    Thus we have

\begin{center}
$\sinh(||[v_k ,v_k '']||) \le \frac{\sinh(d)}{\sinh(d_0)} \cdot \sinh(b)$.
\end{center}

We have shown that the distance from $v_k$ to the hyperplane containing $[v_0 , ... ,v_{k-1}]$ is at most

\begin{center}
$D(b,d_0 ,d) := \operatorname{arcsinh}[\frac{\sinh(d)}{\sinh(d_0)} \cdot \sinh(b)]$.
\end{center}

A similar argument shows the distance from each vertex to the hyperplane containing the opposite face is at most

\begin{center}
$D(b,d_0 ,d)$.
\end{center}

\end{proof}

Next we show that if a bad $k$-simplex has bounded circumradius and good proper simplices, then the vertices lie near a hyperbolic $(k-2)$-sphere.

\begin{lemma}\label{lem3}
Let $S = [v_0 , ... ,v_k]$ be a geodesic $k$-simplex in $\mathbb{H}^k$ with edge lengths in $[a,b]$ and circumradius at most $c$.  If every proper subcomplex of $S$ is $(a,b,d_0)$-good and $S$ is $(a,b,d)$-bad, then the distance from each vertex to the circumsphere of the opposite face is at most a constant $R := R(a,b,c,d_0,d)$ such that $R \rightarrow 0$ as $d \rightarrow 0$ and $a,b,c,d_0$ remain fixed.
\end{lemma}

\begin{figure}[h]
\includegraphics[width=\textwidth]{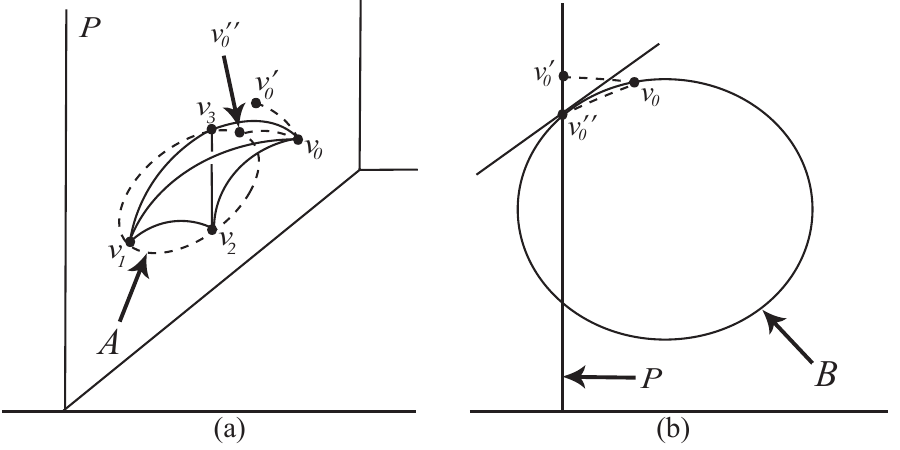}
\caption{The distance from $v_0$ to the circumsphere $A$ of $[v_1 ,...,v_k]$ is small if every proper subcomplex of $S$ is good, but $S$ is bad}
\label{fig2}
\end{figure}

\begin{proof}

Since every proper subcomplex of $S$ is $(a,b,d_0)$-good and $S$ is $(a,b,d)$-bad and has edge lengths in $[a,b]$ and circumradius at most $c$, the distance from each vertex to the hyperplane containing the opposite face is at most the constant $D := D(b,d_0 ,d))$ provided by Lemma \ref{lem2}.

Let $A$ be the circumsphere of $[v_1 , ... ,v_k]$.  Let $B$ be the circumsphere of $S$.  Now $A$ is the intersection of $B$ with some hyperplane $P$.  Let $v_0 '$ be the orthogonal projection of $v_0$ to $P$.  Let $v_0 ''$ be the  point on $A$ which is closest to $v_0$.  Let $Q$ be the $2$-dimensional hyperbolic plane which contains $v_0$, $v_0 '$, and $v_0 ''$.    Now $B \cap Q$ is a hyperbolic circle which intersects the hyperbolic line $P \cap Q$ (see Figure \ref{fig2}).  Since the radii of $A$ and $B$ are in the interval $[a/2 , c]$, the angle between $P \cap Q$ and the tangent of $B \cap Q$ at the two points in $(P \cap Q) \cap (B \cap Q)$ is bounded from below by a positive constant $\alpha_0 := \alpha_0 (a,c)$.  Thus the angle $\alpha$ between $[v_0 '',v_0]$ and $[v_0 '',v_0 ']$ is bounded from below by $\alpha_0$.  Let $\beta$ be the angle between $[v_0 ,v_0 ']$ and $[v_0 ,v_0 '']$.  By the hyperbolic law of sines we have

\begin{align*}
\sinh(||[v_0 ' ,v_0 '']||) &= \frac{\sinh(||[v_0 ,v_0 ']||)}{\sin(\alpha)} \cdot \sin(\beta)\\
                           &\le \frac{\sinh(D)}{\sin(\alpha_0)}.
\end{align*}


\noindent The triangle inequality now gives us

\begin{align*}
||[v_0 ,v_0 '']|| &\le ||[v_0 ,v_0 ']|| + ||[v_0 ' ,v_0 '']|| \\
                   &\le D(b,d_0 ,d) + \operatorname{arcsinh}(\frac{\sinh(D(b,d_0 ,d))}{\sin(\alpha_0 (a,c))}).
\end{align*}

\noindent Let $R(a,b,c,d_0 ,d) := D(b,d_0 ,d) + \operatorname{arcsinh}(\frac{\sinh(D(b,d_0 ,d))}{\sin(\alpha_0 (a,c))})$.

\end{proof}

We can now bound the volume of the bad region of a simplex with bounded circumradius and good proper sub-simplices.

\begin{lemma}\label{lem4}
Let $b > a > 0,c > 0,d_0 > 0,d > 0$  Let $n \ge 3$ and $k < n$.  Let $S$ be a geodesic $k$-simplex in $\mathbb{H}^n$ such that the circumradius of $S$ is at most $c$ and  every proper sub-simplex of $S$ is $(a,b,d_0)$-good.  The volume of the $(a,b,c,d)$-bad region of $S$ is at most a constant $V_k :=V_k(n,a,b,c,d_0 ,d)$ such that $V_k \rightarrow 0$ as $d \rightarrow 0$ and $a,b,c,d_0$ remain fixed.
\end{lemma}

\begin{proof}
The $(a,b,c,d)$-bad region of $S$ is contained in the $R$-neighborhood of the circumsphere $B$ of $S$, where $R := R(a,b,c,d_0,d)$ is the constant provided by Lemma \ref{lem3}.  Since the radius of $S$ is at most $c$ and $R(a,b,c,d_0,d) \rightarrow 0$ as $d \rightarrow 0$, we can let $V_k$ be the volume of the $R$-neighborhood of a $(k-1)$-dimensional hyperbolic sphere of radius $c$ in $\mathbb{H}^n$.
\end{proof}

Let $\delta = \epsilon/10$.  Given a point $p$ in the set $\mathcal{S} \cap M_{[\mu , \infty )}$, the following Lemma bounds the number of $k$-tuples of points which might form a $k$-simplex with $p'$ in a $\delta$-good perturbation $T'$ of the triangulation $T$.

\begin{lemma}\label{lem5}
Let $p \in \mathcal{S} \cap M_{[\mu ,\infty)}$.  For each $k = 3, ... ,n$, the number of $k$-tuples $\{v_1 , ... ,v_k \}$ such that $[p',v_1 ', ... ,v_k ']$ is a $k$-simplex in some $\delta$-good perturbation of $T$ is bounded by a constant $N := N(n,k,\mu)$.
\end{lemma}

\begin{proof}
The $(\frac{\epsilon}{2} -\delta)$-balls centered at the points of $\mathcal{S} \cap M_{[\mu ,\infty)}$ are mutually disjoint since no two points of $\mathcal{S}  \cap M_{[\mu ,\infty)}$ are closer than $\epsilon - 2\delta$ to each other.  If $p'$ and $q'$ are vertices of a $k$-simplex in a $\delta$-good perturbation $\mathcal{T'}$ of $\mathcal{T}$, then $d(p',q') \le 2\epsilon + 2\delta$, so that the $(\frac{\epsilon}{2} -\delta)$-ball centered at $q'$ is contained in the $(2\epsilon + 2\delta)$-ball centered at $p'$.  There can be at most

\begin{center}
$m := m(n,\mu ) = [\frac{\operatorname{Vol}_{\mathbb{H}^n} (B(2\epsilon + 2\delta))}{\operatorname{Vol}_{\mathbb{H}^n} (B(\frac{\epsilon}{2} -\delta))}]$
\end{center}

\noindent mutually disjoint $(\frac{\epsilon}{2} -\delta)$-balls contained in a $(2\epsilon + 2\delta)$-ball, where $[w]$ is the integer part of $w$.  One of these is centered at $p'$.  So there are at most $m - 1$ vertices in $\mathcal{S}$ which may be the vertex of a $k$-simplex in $\mathcal{T'}$ which also has $p'$ as a vertex.  Thus the number of $k$-tuples $\{ v_1 , ... ,v_k \}$ of points in $\mathcal{S}$ such that $[p',v_1 ' , ... ,v_k ']$ is a $k$-simplex in some good perturbation of $\mathcal{T}$ is at most $m\choose{k}$.  Let $N(n,k,\mu ) := {m(n,\mu) \choose{k}}$.
\end{proof}

\begin{lemma}\label{lem6}
A geodesic triangle in $\mathbb{H} ^2$ with edge lengths in $[a,b]$ and circumradius at most $R$ has altitudes bounded from below by a positive constant $h_0 := h_0 (a,b,R)$.
\end{lemma}

\begin{figure}[ht]
\includegraphics[width=\textwidth]{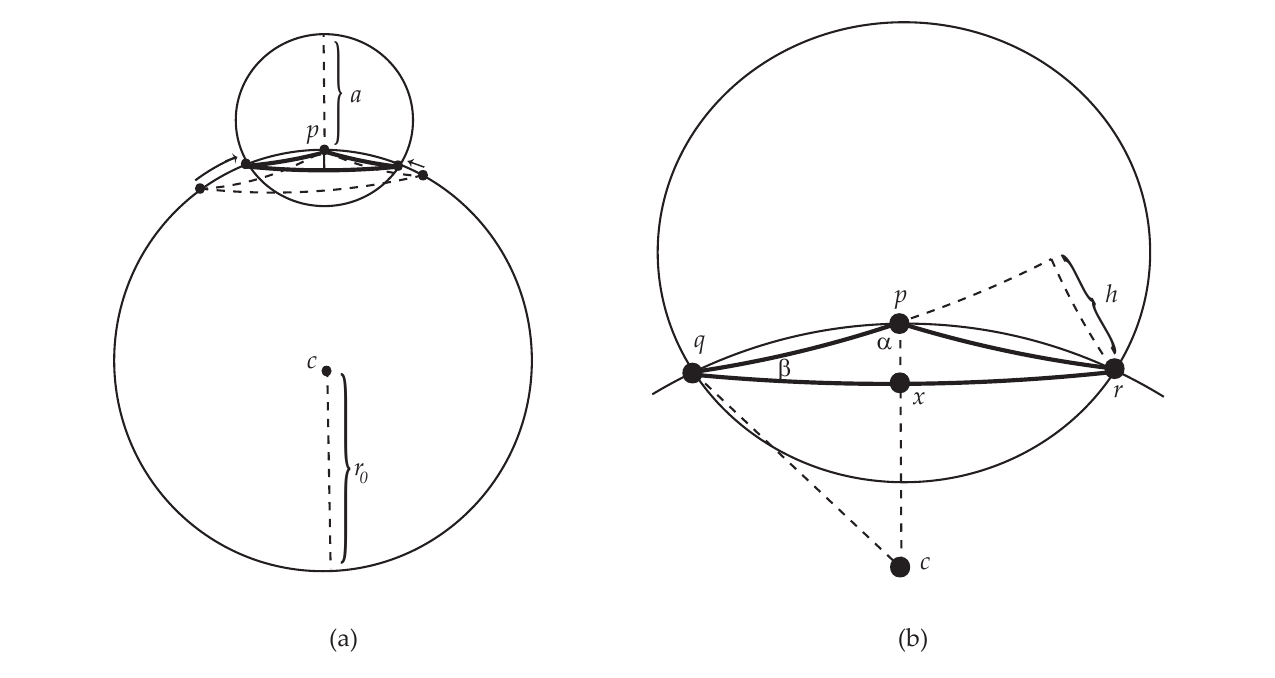}
\caption{(a) If the radius $r_0$ is fixed, then moving $q$ and $r$ closer to $p$ makes the altitude from $p$ smaller.  (b) Both $||[p,x]||$ and $h$ are bounded from below in terms of $a$,$b$, and $R$.}
\label{fig3}
\end{figure}

\begin{proof}

Since the sum of the angles of $t$ are less than $\pi$ there are at least two angles of $t$ which are less than $\pi / 2$.  Let $p$ be the vertex opposite these angles.  The orthogonal projection of $p$ onto the line containing the opposite edge $[q,r]$ is contained in the interior of $[q,r]$.  Now suppose we have fixed the circumradius $r_0 \in [a / 2 ,R]$ of $t$ and consider all triangles with edges of length at least $a$ such that $p$ projects to the interior of $[q,r]$.  The triangle with the shortest altitude at $p$ is an isosceles triangle which lies on a hyperbolic circle of radius $r_0$ (see Figure \ref{fig3}). Let $c$ be the center of the hyperbolic circle containing $p,q,r$.  Let $x$ be the intersection of $[p,c]$ and $[q,r]$.  Let $\alpha = \angle qpx$.  Let $\beta = \angle pqx$.


\noindent By the hyperbolic law of cosines, we have

\begin{align*}
\cos(\alpha ) &= \frac{\cosh(||[p,q]||)\cosh(||[p,c]||) - \cosh(||[q,c]||) }{\sinh(||[p,q]||)\sinh(||[p,c]||)}\\
              &= \frac{\cosh(a)\cosh(r_0 ) - \cosh(r_0 )}{\sinh(a)\sinh(r_0 )}.
\end{align*}

\noindent By the hyperbolic law of sines we have $\sinh(||[q,x]||) =$ $\sinh(||[p,q]||) \sin(\alpha)$.  Now the altitude of $[p,q,r]$ from $p$ is $||[p,x]||$.  Using the law of cosines again, we get

\begin{align*}
\cosh(||[p,x]||) &= \cosh(a)\cosh(||[q,x]||) - \sinh(a)\sinh(||[q,x]||)\cos(\beta )\\
                 &\ge \cosh(a)\cosh(||[q,x]||) - \sinh(a)\sinh(||[q,x]||),
\end{align*}

Let $h_1 (a,r_0 ) = \operatorname{arccosh}(\cosh(a)\cosh(||[q,x]||) - \sinh(a)\sinh(||[q,x]||)$. So far we have shown that the altitude from a vertex which projects to the interior of the opposite face is at least $h_1 (a,r_0 )$ if the circumradius of $[p,q,r]$ is $r_0$.  Since $h_1 (a,r_0)$ decreases as $r_0$ increases, we have that $h_1 (a,R)$ is a lower bound on the altitude from a vertex which projects to the interior of the opposite face for triangles satisfying the hypotheses of the lemma.  Let $h'$ be the altitude from $r$.


\noindent We have

\begin{align*}
\sin(\beta ) &= \frac{\sinh(||[p,x]||)}{\sinh(||[p,q]||)}\\
             &\ge \frac{\sinh(h_1(a,R))}{\sinh(b)}.
\end{align*}

\noindent Also,

\begin{align*}
\sinh(h) &= \sinh(||[q,r]||)\sin(\beta ) \\
         &\ge \sinh(a) \cdot \frac{\sinh(h_1 (a,R))}{\sinh(b)},
\end{align*}

\noindent so that

\begin{center}
$h \ge$ $\operatorname{arcsinh}[\frac{\sinh(a)}{\sinh(b)} \cdot \sinh(h_1 (a,R))]$.
\end{center}

\newpage

\noindent A similar argument works for the altitude from $q$.  Let

\begin{center}
$h_0 (a,b,R) =$ $\operatorname{arcsinh}[\frac{\sinh(a)}{\sinh(b)} \cdot \sinh(h_1 (a,R))]$.
\end{center}

\end{proof}

\noindent\textit{Proof of Theorem.}  The idea of the proof is to show that there is a $\delta$-good perturbation $T'$ of the triangulation $T$ such that each tetrahedron of $T'$ contained in the thick part of $M$ is $(a,b,d)$-good for fixed constants $a,b$, and $d$.  We know each $k$-simplex of $T$ contained in the thick part of $M$ has edge lengths in the interval $[\epsilon ,2\epsilon ]$ and circumradius at most $\epsilon$ (for $k = 1, ...,n$).  Note that if $t$ is a $k$-simplex in the Delaunay triangulation of a $\delta$-good perturbation of $\mathcal{S}$ which is contained in $M_{[\mu ,\infty )}$, then $t$ has edge lengths between $\epsilon - 2\delta$ and $2\epsilon + 2\delta$, and circumradius no more than $\epsilon + \delta$.

We will remove the bad simplices one dimension at a time.  Let $a := \epsilon - 2\delta$, $b := \epsilon + 2\delta$, $c := \epsilon + \delta$.  We will proceed by induction on the dimension of the simplices.

Since any 2-simplex in $T$ has edge lengths in $[a,b]$ and circumradius at most $c$ (where $a := \epsilon - 2\delta$, $b := \epsilon + 2\delta$, $c := \epsilon + \delta$), Lemma \ref{lem6} implies that each 2-simplex in $T$ is $(a,b,h_0 (a,b,c))$-good.

Assume that $T_k$ is a $\frac{\delta}{100 \cdot 2^k}$-good perturbation of $T$ such that every simplex of dimension at most $k$ which is contained in the $\mu$-thick part of $M$ is $(a,b,d_k)$-good for some positive constant $d_k \le d_2$  Let $\delta _{k+1} := \frac{\delta}{100 \cdot 2^{k+1}}$.  Let $d_{k+1} \le d_k$ be a positive constant to be determined later.

Let ${p_{1}}\in\mathcal{S}\cap\ M_{[\mu,\infty)}$.  Let ${\mathcal{U}}_{1}$ be the set simplices $[v_0 , ... ,v_l]\in\mathcal{T}_k$ of dimension at most $k$ (i.e. $l \le k$) such that there exists a $\delta_{k+1}$-good perturbation $\mathcal{T'}$ of $\mathcal{T}$ which is obtained by perturbing only the point $p_{1}$ and such that $[{p_{1}}',v_0 , ... ,v_l]\in\mathcal{T'}$.

By Lemma \ref{lem4} and Lemma \ref{lem5}, the total volume of the $(a,b,c,d_{k+1})$-bad regions of the $l$-simplices in ${\mathcal{U}}_1$ is bounded by $V_l (a,b,c,d_2 ,d_{k+1}) \cdot N(n,l,\mu)$, so that the total volume of the $(a,b,c,d_{k+1})$-bad regions of all simplices in $\mathcal{U} _{1}$ is bounded by $\sum_{l=1}^{k+1} V_l (a,b,c,d_2 ,d_{k+1}) \cdot N(n,l,\mu)$.  Let $B(p,\delta_{k+1} )$ be the ball of radius $\delta_{k+1}$ centered at $p.$ If we choose $d_{k+1}$ so small that
$\sum_{l=1}^{k+1} V_l (a,b,c,d_2 ,d_{k+1}) \cdot N(n,l,\mu) \le vol(B(p_1 ,\delta_{k+1} )  )$,  then the $(a,b,c,d_{k+1})$-bad regions of the simplices in $\mathcal{U} _1$ cannot cover $B(p_1 ,\delta _{k+1} )$.  Now choose ${p_{1}}'$ in $B(p_1 ,\delta_{k+1} )$ (so that the perturbation is $\delta_{k+1}$-good) and outside the $(a,b,c,d_{k+1})$-bad region of every simplex in ${\mathcal{U}}_{1}$.  Call the new set of points ${\mathcal{S}}_{1}$ and the new triangulation ${\mathcal{T}}_{1}$.

Assume we have perturbed the points $p_{1},...,p_{s}$ to ${p_{1}}',...,{p_{s}}'$ and now have a set of points ${\mathcal{S}}_{s}$ and a triangulation ${\mathcal{T}}_{s}$ such that none of ${p_{1}}',...,{p_{s}}'$ is the vertex of a $(a,b,c,d_{k+1} )$-bad simplex of dimension less than $k+2$.  Let $p_{s+1}$ be a point in $[{\mathcal{S}}_{s} \cap M_{[\mu,\infty)}]     -\{{p_{1}}',...,{p_{s}}'\}$.  Let ${\mathcal{U}}_{s+1}$ be the set of simplices $[v_0 , ... ,v_l]\in{\mathcal{T}}_{s}$ of dimension at most $k$ such that there exists a $\delta_{k+1}$-good perturbation ${\mathcal{T}_{s}}'$ of ${\mathcal{T}}_{s}$ which is obtained by perturbing only the point $p_{s+1}$ and such that $[{p_{s+1}}',v_0 , ... , v_k]\in{{\mathcal{T}_{s}}'}$.  Since $d_{k+1}$ is so small, we can choose a point ${p_{s+1}}'$ in the ball of radius $\delta_{k+1}$ centered at $p_{s+1}$ and outside the $(a,b,c,d_{k+1})$-bad region of every simplex in ${\mathcal{U}}_{s+1}$.

Assume that $M$ has finite volume. Let $\mathcal{T'}$ be the triangulation we get after perturbing every point of $\mathcal{S} \cap M_{[\mu ,\infty )}$ once and only once (There are only finitely many since $M$ has finite volume).  Let $[p',v_1 ' , ... ,v_l ']\in\mathcal{T'}$ be a simplex of dimension at most $k+1$.  Suppose that $p'$ was the last point perturbed among these $l+1$ vertices.  We chose $p'$ to be outside the $(a,b,c,d_{k+1})$-bad region of $[v_1 ' , ... ,v_l ']$, so that $[p', v_1 ' , ... , v_l ']$ is $(a,b,d_{k+1})$-good.  Thus any $(k+1)$-simplex of $\mathcal{T'}$ contained in $M_{[\mu, \infty)}$ is $(a,b,d_{k+1})$-good.

If $M$ has infinite volume, then for each positive integer $m$ the above procedure can be used to perturb the vertices contained in an $m$-ball centered at some fixed point $x_0$, giving us a geodesic triangulation $\mathcal{T}(m)$ of $M$ such that any tetrahedron contained in $M_{[\mu, \infty)} \cap B(x_0,m)$ is $(a,b,d_{k+1})$-good.  Suppose we want to define the final triangulation on a ball $B(x_0,N)$ for some positive integer $N$.  Since the triangulations $\mathcal{T}(m)$ agree on the ball $B(x_0,N)$ for $m\ge 100N$, we can use the triangulation $\mathcal{T}(100N)$ to define the triangulation inside $B(x_0,N)$.

We have shown that $M$ has a geodesic triangulation such that every simplex of dimension at most $k+1$ which is contained in the thick part of $M$ is $(a,b,c,d_{k+1})$-good, completing the induction.  When $k=n-1$, we get a geodesic triangulation of $M$ such that every simplex of dimension at most $n$ which is contained in the thick part of $M$ is $(a,b,d_{n})$-good. Thus the triangulation is $L$-thick for a constant $L$ depending only on $a$, $b$, and $d_{n}$, which depend only on $\mu$ and $n$. \hfill $\Box$\\

\textbf{Acknowledgement.}  This work was partially supported by the NSF grant DMS-0135345.\\

\bibliographystyle{amsalpha}
\bibliography{tri}

\end{document}